\documentclass[english,11pt]{amsart}
\usepackage{amsopn,amsthm,amsfonts,amsmath,amssymb,amscd}
\usepackage{babel}
\usepackage{amssymb,tikz}

\newtheorem{Theorem}{Theorem}[section]
\newtheorem{Lemma}[Theorem]{Lemma}
\newtheorem{Proposition}[Theorem]{Proposition}

\newtheorem{Example}[Theorem]{Example}
\newtheorem*{Remark}{Remark}
\newtheorem{Def}[Theorem]{Definition}

\newenvironment{Proof*}{{\it Proof.}}

\newcommand{\BB}{\mathcal{B}}

\newcommand{\DEF}[1]{\emph{#1}}

\newcommand{\supp}{{\rm supp}}
\newcommand{\diam}{{\rm diam}}

\begin{document}

\title{The metric dimension of the total graph of a semiring}
\author{David Dol\v zan}
\date{\today}

\address{D.~Dol\v zan:~Department of Mathematics, Faculty of Mathematics
and Physics, University of Ljubljana, Jadranska 19, SI-1000 Ljubljana, Slovenia, and Institute of Mathematics, Physics and Mechanics, Jadranska 19, SI-1000 Ljubljana, Slovenia; e-mail: 
david.dolzan@fmf.uni-lj.si}

  \thanks{The author acknowledges the financial support from the Slovenian Research Agency  (research core funding No. P1-0222)}

\bigskip

\begin{abstract} 
We calculate the metric dimension of the total graph of a direct product of finite commutative antinegative semirings with their sets of zero-divisors closed under addition.
\end{abstract}

\maketitle 
%\parindent=0cm

%-----------------------------------------------------
%-----------------------------------------------------
\section{Introduction}
%-----------------------------------------------------
%-----------------------------------------------------

\bigskip

Of the many ways one can prescribe a graph to an algebraic structure, one of the most important ones is the total graph. The total graph of a ring or a semiring $S$ is a graph with all elements of $S$ as vertices, and for distinct $x, y \in S$, the vertices $x$ and $y$ are adjacent if and only if $x + y$ is a zero-divisor in $S$. Total graph has been firstly studied in \cite{anderson} for commutative rings and in \cite{oblak} for non-commutative rings. Later, it has been studied in \cite{oblak1, atani} in the setting of commutative semirings and in \cite{john} in the case of matrices over the Boolean semiring.

In a simple undirected graph, there exist a couple of natural ways to distinguish the vertices from one another - such ideas appear already in Sumner \cite{sumner} and Entringer and Gassman \cite{entringer}. Among them, one of the most important ones is the metric dimension of a graph. Let us explain the definition.
Choose an ordered subset $W=\{ w_1,w_2,\ldots,w_k\}$ of the vertex set of graph $G$. Then, for a vertex $v$ in $G$, the vector $(d(v,w_1),d(v,w_2),\ldots,d(v,w_k))$ is called the \DEF{representation} of $v$ with respect to $W$. A set $W$ is called a \DEF{resolving set} for $G$ if distinct vertices of $G$ have distinct representations with respect to $W$. A resolving set of minimal cardinality for $G$ is called a \DEF{basis} of $G$ and the cardinality of the basis is defined as the \DEF{metric dimension} of $G$, and denoted by $\dim(G)$. 

The metric dimension was studied by Slater \cite{slater}, while trying to determine the location of an intruder in a network. The notion was then also studied by Harary and Melter \cite{harary}. The metric dimension has later been used in many applications, for example in pharmaceutical chemistry \cite{cameron}, robot navigation in space \cite{khuller}, combinatorial optimization \cite{sebo} and sonar and coast guard long range navigation \cite{slater, slater2002}.

%The second approach was first considered by Tutte \cite{tutte} in view of enumerating plane maps. Here, a subset of the vertex set is distinguished in such a way that the automorphism group of the graph is destroyed, i.e., the automorphism group of the resulting structure is trivial. The notion was later studied by Albertson and Collins \cite{albertson}, and Erwin and Harary \cite{erwin} who defined the notion of a fixing number (denoted by $\fix(G)$), which is the cardinality of smallest subset $W'$ such that the only graph automorphism on $G$ that is identical on $W'$, is the identity automorphism. The results in this area of research have been applied to the problem of programming a robot to manipulate objects \cite{lynch}. It turns out that the two approaches have a close connection, as $\fix(G) \leq \dim(G)$ for every connected graph $G$ (\cite[Theorem 3]{erwin}) and an interesting question arises naturally, namely for which graphs are these two notions identical. 

In view of studying the metric dimensions of graphs corresponding to different algebraic structures, the metric dimension of a zero-divisor graph of a commutative ring was studied in \cite{ou, pirzada1, pirzada, raja, redmond21}, a zero-divisor graph of a matrix semiring in \cite{dolzan0}, a total graph of a finite commutative ring in \cite{dolzan1}, an annihilating-ideal graph of a finite ring in \cite{dolzan21}, and a commuting graph of a dihedral group in \cite{ali}. 
%The fixing number has also recently been studied in this context. For example, the fixing number of a zero-divisor graph of a commutative ring has been studied in \cite{ou}.

In this paper, we study the metric dimension of the total graph of a direct product of finite commutative antinegative semirings. A \emph{semiring} is a set $S$ equipped with binary operations $+$ and $\cdot$ such that $(S,+)$ is a commutative monoid with identity element 0 and $(S,\cdot)$ is a monoid with identity element 1, while both operations in $S$ are connected by distributivity. We also assume that $0$ annihilates $S$. Semiring $S$ is called \emph{commutative} if we have $ab=ba$ for all $a,b \in S$ and \emph{antinegative} or \emph{zero-sum-free} (or simply, an \emph{antiring}), if $a+b=0$ for $a, b \in S$ implies $a=b=0$. Commutative semirings often arise naturally, since the set of nonnegative integers (or reals) with the usual operations of addition and multiplication forms a commutative semiring. But there are many more examples, as distributive lattices, tropical semirings, dioids, fuzzy algebras, inclines and bottleneck algebras are all different examples of commutative semirings. The simplest semiring is the Boolean semiring $\BB=\{0,1\}$, where $1+1=1$. The theory of semirings also provides many applications, for example in optimization theory, automatic control, models of discrete event networks and graph theory (\cite{baccelli, cuninghame, li, zhao}). 

This paper is organized as follows. In the next section, we deal with the preliminary concepts and lemmas that come useful later, while we also make an argument as to why it makes more sense to study the total graph of the semiring, where the vertex set is limited to the set of all zero-divisors. In the main section, we determine the metric dimension of a direct product of antinegative finite commutative semirings, with their set of zero-divisors closed under addition. Such semirings for example contain the Boolean semiring, all finite distributive lattices with zero and exactly one atom, all finite semifields, all polynomial semirings over finite antirings modulo a monome, etc.
It turns out that the metric dimension behaves somewhat differently in different settings, with the exceptions consting of semifields and semirings such that all their elements except the identity are zero divisors. The main results of this paper are Proposition \ref{semimetricbool}, Theorem \ref{semimetricu} and Theorem \ref{haupt}, which together cover all possible cases.
It is perhaps interesting to compare these results with \cite[Theorem 3.5 and Theorem 3.7]{dolzan1}, where even for the direct product of two finite fields, the metric dimension of its total graph is not completely determined.

\bigskip

%-----------------------------------------------------
%-----------------------------------------------------

\bigskip
\section{Preliminaries}
\bigskip

Throughout this paper, we shall suppose that $S$ is a semiring, and that $Z(S)$ denotes the set of all zero-divisors in $S$, $Z(S)=\{x \in S;$ there exists $0 \neq y \in S \text { such that } xy=0 \text { or } yx=0 \}$. %The group of all invertible elements in $S$ will be denoted by $S^*$.
We denote by $\Gamma(S)$ the \emph{total graph} of $S$. 
The vertex set $V(\Gamma(S))$ of $\Gamma(S)$ is the set of all elements in $S$ and 
an unordered pair of vertices $x,y \in V(\Gamma(S))$, $x \neq y$, is an edge 
$x-y$ in $\Gamma(S)$ if $x+y \in Z(S)$.
The sequence of edges $x_0 - x_1$, $ x_1 - x_2$, ..., $x_{k-1} - x_{k}$ in $\Gamma$ is called \emph{a path of length $k$}. The \DEF{distance} between vertices $x$ and $y$ is the length of the shortest
path between them, denoted by $d(x,y)$, while the \DEF{diameter} $\diam(\Gamma)$ of the graph $\Gamma$ is the longest distance between any two vertices of the graph.

We shall need the following definition.

\begin{Def}
Let $v$ be a vertex of a graph $G$. Then the \emph{open neighbourhood} of $v$ is $N(v)=\{u \in V(G); \text{ there exists an edge } uv \text { in G}\}$ and the \emph{closed neighbourhood} of $v$ us $N[v]=N(v) \cup \{v\}$.  
Two distinct vertices $u$ and $v$ of $G$ are \emph{twins} if $N(u)=N(v)$ or $N[u]=N[v]$.
\end{Def}

Twins turn out to be a very important notion when studying the metric dimension of a graph due to the following lemma. 

\begin{Lemma}[\cite{hernando}]\label{twins}
Suppose $u$ and $v$ are twins in a connected graph $G$ and the set $W$ is a resolving set for $G$. Then $u$ or $v$ is in $W$.
\end{Lemma}

The next lemma shows that in order to study the metric dimension, it makes sense to limit ourselves to the subgraph with the vertex set equal to the set of all zero-divisors.

\begin{Lemma}
\label{semidisco}
Let $S$ be an antinegative semiring. Then $\Gamma(S)$ is a disconnected graph (where the set of all non zero-divisors are isolated vertices). The subgraph $\Gamma_1(S)$ with the vertex set equal to $Z(S)$ is connected with $\diam(\Gamma_1(S)) \leq 2$.
\end{Lemma}
\begin{proof}
Choose and $s \in S$ such that $s \notin Z(S)$ and suppose that $s - w$ is an edge in $\Gamma(S)$ for some $w \in S$. This implies that $s+w \in Z(S)$, so there exists $z \neq 0$ such that $(s+w)z=0$. Since $S$ is antinegative, this yields $sz=0$, a contradiction. 
On the other hand, if $s \in Z(S)$, then $0 - s$ is an edge in $\Gamma_1(S)$. 
\end{proof}

Therefore, we shall limit ourselves to studying the subgraph $\Gamma_1(S)$ from here onwards. 
We shall now investigate twins in the graph $\Gamma_1(S)$. 

\begin{Lemma}
\label{whentwins}
Let $n \geq 1$ and $S=S_1 \times S_2 \times \ldots \times S_n$, where $S_i$ is a finite antinegative commutative semiring with $Z(S_i)$ closed under addition for every $i=1,2,\ldots,n$. Choose $a=(a_1,a_2,\ldots,a_n), b=(b_1,b_2,\ldots,b_n) \in Z(S)$. Then $a$ and $b$ are twins in $\Gamma_1(S)$ if and only if either $a_i,b_i \in Z(S_i)$ or $a_i,b_i \notin Z(S_i)$ for every $i=1,2,\ldots, n$.
\end{Lemma}
\begin{proof}
Choose any $1 \leq i \leq n$.
Notice that the fact that $S_i$ is antinegative implies that $a_i+b_i \in Z(S_i)$ yields $a_i, b_i \in Z(S_i)$: if there exists $x \in S_i$ such that $0=(a_i+b_i)x=a_ix+b_ix$, we have $a_ix=b_ix=0$.
Since $Z(S_i)$ closed under addition for every $i=1,2,\ldots,n$, we also have that $a_i, b_i \in Z(S_i)$ implies $a_i+b_i \in Z(S_i)$.
So, there is an edge between $a=(a_1,a_2,\ldots,a_n)$ and $c=(c_1,c_2,\ldots,c_n)$ in $\Gamma_1(S)$ if and only if there exists $1 \leq i \leq n$ such that $a_i$ and $c_i$ are both in $Z(S_i)$. Thus, the neighbourhood of any vertex is completely determined by which of its components belong to the set of zero-divisors. 
Therefore, two elements $a=(a_1,a_2,\ldots,a_n)$ and $b=(b_1,b_2,\ldots,b_n)$ in $S$ are twins if and only if either $a_i,b_i \in Z(S_i)$ or $a_i,b_i \notin Z(S_i)$ for every $i=1,2,\ldots, n$.
\end{proof}

\bigskip

%-----------------------------------------------------
%-----------------------------------------------------
\section{Metric dimension}
%-----------------------------------------------------
%-----------------------------------------------------

\bigskip

In this section, we examine the metric dimension of a total graph of a direct product of antinegative commutative semirings. 
Firstly, let us establish the lower bound for the metric dimension.

\begin{Lemma}
 \label{lowerb}
Let $n \geq 1$ and $S=S_1 \times S_2 \times \ldots \times S_n$, where $S_i$ is a finite antinegative commutative semiring with $Z(S_i)$ closed under addition for every $i=1,2,\ldots,n$. Then
$\dim(\Gamma_1(S)) \geq |Z(S)|-2^n+1$.
\end{Lemma}
\begin{proof}
Suppose that $W$ is a resolving set for $\Gamma_1(S)$ and choose a subset $\emptyset \neq X \subseteq \{1,2,\ldots,n\}$. Define $M_X=M_1 \times M_2 \times \ldots \times M_n \subset Z(S)$, where $M_i=Z(S_i)$ if $i \in X$ and $M_i=S_i \setminus Z(S_i)$ if $i \notin X$. By Lemma \ref{whentwins}, all elements of $M_x$ are twins, so by Lemma \ref{twins}, $W$ contains all but perhaps one element of the set $M_X$. 
Observe that $a=(a_1,a_2,\ldots,a_n) \in S$ is a vertex in $\Gamma_1(S)$ if and only if there exists $i \in \{1,2,\ldots,n\}$ such that $a_i \in Z(S_i)$. Therefore, every vertex in $\Gamma_1(S)$ lies in exactly one of the sets $M_X$ for a chosen set $X \neq \emptyset$. So, $W$ contains all elements of $Z(S)$ apart from perhaps as many as there are nontrivial subsets of the set $\{1,2,\ldots,n\}$. Since there are exactly $2^n-1$ such nontrivial subsets, we have
$|W| \geq |Z(S)|-2^n+1$. Because this is true for any resolving set $W$, 
we have $\dim(\Gamma_1(S)) \geq |Z(S)|-2^n+1$ as stated.
\end{proof}

Let us examine the interesting special case of the metric dimension of the total graph of the direct product of a number of Boolean semirings.

\begin{Remark}
If $S$ is an antinegative finite commutative semiring with $|S| \geq 3$ and $|Z(S)| = 1$, then $\Gamma_1(S)$ is a trivial graph (since it has only one vertex) - obviously, $\dim(\Gamma_1(S))=0$ in this case. We shall therefore from here onwards always limit ourselves to the case where $\Gamma_1(S)$ is non-trivial.
\end{Remark}

\begin{Proposition}
\label{semimetricbool}
Let $n \geq 2$ and $S=S_1 \times S_2 \times \ldots \times S_n$, where $S_i=\BB$ for every $i=1,2,\ldots,n$. Then
$\dim(\Gamma_1(S))=n$ for $n \neq 2$ and $\dim(\Gamma_1(S))=n-1$ for $n = 2$.
\end{Proposition}
\begin{proof}
Assume first that $n \geq 3$.
For every $i=1,2,\ldots,n$, let us denote by $\overline{e_i}$ the element with $1$ at every component, except the $i$-th component, which is equal to $0$. 
Let $W=\{\overline{e_1},\overline{e_2},\ldots,\overline{e_n}\}$. Now, choose distinct $x=(x_1,x_2,\ldots,x_n)$ and $y=(y_1,y_2,\ldots,y_n)$ in $S$. There exists $i \in \{1,2,\ldots,n\}$ such that $x_i \neq y_i$. Without loss of generality, we can assume $x_i=0$ and $y_i=1$. Then $x - \overline{e_i}$ is an edge in $\Gamma(S_1)$ and $y - \overline{e_i}$ is not, therefore $x$ and $y$ have different representations with respect to $W$. Thus, $W$ is a resolving set, which proves that $\dim(\Gamma_1(S)) \leq n$. On the other hand, Lemma \ref{semidisco} tells us that $\diam(\Gamma_1(S)) \leq 2$, so any representation of a vertex that is not a member of a resolving set, has to be a vector with every component in $\{1,2\}$. Assume that $W$ is a resolving set and denote $w=|W|$. Since there are at most $2^w$ different representations of vertices that are not in $W$ and there are $2^n-1$ vertices in $\Gamma_1(S)$, we have $2^w+w \geq 2^n-1$. Obviously, $w \mapsto 2^w+w$ is an increasing function, and $2^{n-1}+n-1 \geq 2^n-1$ implies $n \geq 2^{n-1}$, which further implies $n \leq 2$. Therefore for $n \geq 3$, we can conclude that $w \geq n$. So, we have proved that $\dim(\Gamma_1(S)) \geq n$ for all $n \geq 3$ and thus $\dim(\Gamma_1(S)) = n$.
Finally, let us examine the case $n=2$. Then $S=\BB \times \BB$ and one can easily check that $w=\{(1,0)\}$ is a resolving set, so $\dim(\Gamma_1(S)) = 1$.
\end{proof}

We can now prove the following theorem.

\begin{Theorem}
\label{semimetricu}
Let $n \geq 1$ and $S=S_1 \times S_2 \times \ldots \times S_n$, where $S_i$ is a finite antinegative commutative semiring with $Z(S_i)$ closed under addition for every $i=1,2,\ldots,n$. Furthermore, assume that for every $i=1,2,\ldots,n$ the following condition holds: $|Z(S_i)| \geq 2$  or there exists $j \neq i$ such that $|S_j \setminus Z(S_j)| \geq 2$. Then
$\dim(\Gamma_1(S))=|Z(S)|-2^n+1$.
\end{Theorem}
\begin{proof}
By Lemma \ref{lowerb}, we already know that $\dim(\Gamma_1(S)) \geq |Z(S)|-2^n+1$.
Choose any subset $\emptyset \neq X \subseteq \{1,2,\ldots,n\}$ and define $M_X=M_1 \times M_2 \times \ldots \times M_n \subset Z(S)$, where $M_i=Z(S_i)$ if $i \in X$ and $M_i=S_i \setminus Z(S_i)$ if $i \notin X$.  Now, choose $t_X \in M_X$ and denote $M'_X = M_X \setminus \{ t_X\}$. If $M'_X \neq \emptyset$, choose $s_X \in M'_X$. Define $W=\bigcup_{\emptyset \neq X \subseteq \{1,2,\ldots,n\}}M'_X$ and let us prove that $W$ is a resolving set for $\Gamma_1(S)$. Choose any distinct elements $u=(u_1,u_2,\ldots,u_n),v=(v_1,v_2,\ldots,v_n) \in Z(S)$ such that $u,v \notin W$. 
By the construction of $W$ and Lemma \ref{whentwins}, there exists $i \in \{1,2,\ldots,n\}$ such that $u_i \in Z(S_i)$ and $v_i \notin Z(S_i)$ or vice versa. Without any loss of generality, we can assume that $u_i \in Z(S_i)$ and $v_i \notin Z(S_i)$. Define $Y=\{i\}$ and observe $M'_Y$ is nonempty by our assumption. Therefore, $u - s_Y$ is an edge in $\Gamma(S_1)$ since $Z(S_i)$ closed under addition, while $v - s_Y$ is not an edge in $\Gamma(S_1)$ since $S_i$ is an antiring for every $i=1,2,\ldots,n$. Thus, $u$ and $v$ have different representations with respect to $W$ and therefore $W$ is a resolving set.
We can use the same proof as in the proof of Lemma \ref{lowerb}, to see that $|W|=|Z(S)|-2^n+1$, which now implies that $\dim(\Gamma_1(S))=|Z(S)|-2^n+1$.
\end{proof}

Let us illustrate the above with an example.

\begin{Example}
Suppose $S=S_1 \times S_2 \times \ldots \times S_n$, where $n \geq 2$ and $S_i$ is a finite antinegative semifield for every $i=1,2,\ldots,n$. Suppose that at least two of the semifields $S_i$ are not isomorphic to $\BB$. Then all the assumptions of Theorem \ref{semimetricu} are satisfied, so $\dim(\Gamma_1(S))=|Z(S)|-2^n+1=|S_1||S_2|\ldots|S_n|-(|S_1|-1)(|S_2|-1)\ldots(|S_n|-1)-2^n+1$.

On the other hand, if $S=S_1 \times S_2 \times \ldots \times S_n$, where $S_i=\BB$ for every $i=1,2,\ldots,n$. Then by Proposition \ref{semimetricbool},
$\dim(\Gamma_1(S))=n$ for $n \neq 2$ and $\dim(\Gamma_1(S))=n-1$ for $n = 2$.

Compare these results with results in \cite[Theorem 3.5 and Theorem 3.7]{dolzan1}, where even for the direct product of two finite fields, the metric dimension of the total graph is not known in all cases. Also, it is interesting to compare these results with results in \cite[Theorem 6.1]{raja}, where similar results have been obtained, but in a different setting of the zero-divisor graphs of finite commutative rings, where again the field with two elements is an exception to the rule.
\end{Example}

Let us further investigate the other possible cases that are not covered by Proposition \ref{semimetricbool} and Theorem \ref{semimetricu}.
Before we do that, we need the following definition.

\begin{Def}
 Let $S=S_1 \times S_2 \times \ldots \times S_n$ a semiring. Then for every $i=1,2,\ldots,n$, define $W_i=\{0,1\} \subseteq S_i$. Then $\supp(S)=W_1 \times W_2 \times \ldots \times W_n \subseteq S$ is a \emph{support} of $S$.    
\end{Def}

\begin{Theorem}
\label{haupt}
\begin{enumerate}
   \item
   Let $S=B_1 \times B_2 \times \ldots \times B_m \times Z_1 \times Z_2 \times \ldots \times Z_n$, where $m, n \geq 1$, $B_i=\BB$ for $i=1,2,\ldots,m$, and $Z_i$ is an antinegative finite commutative semiring with $Z(Z_i)$ closed under addition, $|Z_i| \geq 3$, $|Z_i \setminus Z(Z_i)|=1$ for every $i=1,2,\ldots,n$.
   Then $\dim(\Gamma_1(S)) = |S|+m-2^{m+n}=|Z(S)|-2^{m+n}+m+1$.
   
    \item 
    Let $S=R \times B_1 \times B_2 \times \ldots \times B_m \times Z_1 \times Z_2 \times \ldots \times Z_n$, where $m, n \geq 0$ and $mn \neq 0$, $R$ is an antinegative finite commutative semiring with $|R| \geq 3$ and $|Z(R)|=1$, $B_i=\BB$ for $i=1,2,\ldots,m$, and $Z_i$ is an antinegative finite commutative semiring with $Z(Z_i)$ closed under addition, $|Z_i| \geq 3$, $|Z_i \setminus Z(Z_i)|=1$ for every $i=1,2,\ldots,n$.
    Then $\dim(\Gamma_1(S)) = |S|-2^{m+n+1}-|R|+3=|Z(S)|-2^{m+n+1}+2$.
\end{enumerate}
\end{Theorem}
\begin{proof}
  Define $T = \{v \in Z(S); $ there exists $v \neq u \in Z(S)$ such that $u$ and $v$ are twins in $\Gamma_1(S)\}$, $N = Z(S) \setminus T$ and $T'=T \setminus \supp(S)$. Now, let us look at both cases separately.
  \begin{enumerate}
  \item
  Obviously, $N$ contains exactly all elements from $S$ that have no twins. It follows from Lemma \ref{whentwins} that $N$ contains exactly all elements from $Z(S)$ of the form $(b_1, b_2, \ldots, b_m, 1, 1, \ldots, 1)$ for every $i=1,2,\ldots,m$. Since $T$ is a complement of $N$ in $Z(S)$,
    observe that then $T'$ contains exactly all elements of the form $(b_1, b_2, \ldots, b_m, z_1, z_2, \ldots, z_2)$ for all $b_i \in B_i$ ($1 \leq i \leq m$) and all $z_j \in Z_j$ ($1 \leq j \leq n$), except for those that have $z_j \in \{0,1\}$ for every $j$. Therefore $|T'|=|S|-2^{n+m}$. 
   
    By Lemma \ref{twins}, every resolving set $W$ has to contain at least $|T'|$ elements (and we can limit ourselves to the case where $W$ contains $T'$ without any loss of generality, since in every $Z_i$ any zero-divisor is interchangeable with $0$ in the sense that it has the same neighbourhood). 
    
    Now, define $X=\{(b_1,b_2,\ldots,b_m,0,0,\ldots,0); b_i \in B_i$ for every $i=1,2,\ldots,m \}$ and observe that $|X|=2^m$. Note that by definition, $X \cap T' = \emptyset$ and also that $d(x,t)=d(x',t)=1$ for every $x,x' \in X$ and every $t \in T'$. If we want to have different representations for different elements from $X$, we therefore have to add some elements of the form $(b_1,b_2,\ldots,b_m,1,1,\ldots,1)$ to $T'$, since only these kind of elements can be at distance more than one from elements in $X$. We now reason similarly, as in the proof of Proposition \ref{semimetricbool}.  Suppose we add $w$ elements to $T'$. By Lemma \ref{semidisco}, this means that we have at most $2^w$ possible different new representations of vertices from $X$. There are $2^m$ vertices in $X$, which implies $2^w+w \geq 2^m$. Obviously, $w \mapsto 2^w+w$ is an increasing function, and $2^{m-1}+m-1 \geq 2^m$ implies $m-1 \geq 2^{m}$, which is a contradiction, so $w \geq m$. We have therefore proved that $\dim(\Gamma_1(S)) \geq |S|+m-2^{n+m}$. 
    
    On the other hand, for every $i=1,2,\ldots,n+m$ denote by $\overline{e_i}(x)$ the element of $S$ with $1$ at every component, except the $i$-th component, which is equal to $x$. Define $T''=\{\overline{e_1}(0),\overline{e_2}(0),\ldots,\overline{e_m}(0)\}$ and $W=T' \cup T''$ for $T'$ as above. Observe that $|W|=|S|+m-2^{n+m}$. Let us prove that $W$ is a resolving set. Choose $b=(b_1,b_2,\ldots,b_m,z_1,z_2,\ldots,z_n) \neq b'=(b_1',b_2',\ldots,b_m',z_1',z_2',\ldots,z_n') \in Z(S) \setminus W$ for some $b_i, b_i' \in B_i$ and some $z_j, z,j' \in Z_j$ for every $1 \leq i \leq m, 1 \leq j \leq n$. Suppose firstly that there exists $1 \leq j \leq n$ such that $z_j \neq z_j'$. Since $b, b' \notin W$, we know that $z_j=0$ and $z_j'=1$ or vice versa. Without loss of generality, we can assume the former.
    Now, choose $x \in Z(S_j) \setminus \{0\}$ and note that $\overline{e_j}(x) \in T'$. This now implies that $d(b, \overline{e_j}(x))=1$, while $d(b', \overline{e_j}(x)) > 1$.
    Suppose now that $z_j = z_j'$ for all $1 \leq j \leq n$.
    Since $b \neq b'$, there exists $1 \leq i \leq m$ such that $b_i \neq b_i'$. Then we have $d(b,\overline{e_i}(0)) \neq d(b',\overline{e_i}(0))$.
    This implies that $W$ is indeed a resolving set, so $\dim(\Gamma_1(S)) \leq |S|+m-2^{n+m}$, which now proves our assertion. Note also that every element of $S$ except for $(1,1,\ldots,1)$ is a zero divisor, so 
    $|S|+m-2^{n+m}=|Z(S)|-2^{m+n}+m+1$.

 \item
   Since $N$ consists of exactly all those elements from $Z(S)$ that have no twins, Lemma \ref{whentwins} this time shows that $N$ contains exactly all the elements from $Z(S)$ of the form $(0, b_1, b_2, \ldots, b_m, 1, 1, \ldots, 1)$, where $b_i \in B_i$ for every $i=1,2,\ldots,m$. Again, $T$ is the complement of $N$ in $Z(S)$, so
  we can reason similarly as above that $|T'|=|S|-2^{m+n+1}-(|R|-2)$, since in this case we have to exclude from $S$ all the elements of the form $(r, b_1, b_2, \ldots, b_m, z_1, z_2, \ldots, z_2)$ for all $r \in \{0,1\}$, all $b_i \in B_i$ and all $z_j \in \{0,1\}$, as well as $|R|-2$ elements of the form $(r, 1, 1, \ldots, 1)$ for $r \notin \{0,1\}$ since they are not zero-divisors at all.

  Suppose now that $W$ is a resolving set. Using Lemma \ref{twins}, we can reason that $W$ has to contain at least $|T'|$ elements. We can also limit ourselves to the case where $W$ contains $T'$ without any loss of generality, since in every $Z_i$ any zero-divisor is interchangeable with $0$ and in $R$ every non-zero element is interchangeable with $1$ in the sense that they have the same neighbourhood. 

  Let us assume that $W=T'$ and
  choose $b=(1,0,0,\ldots,0)$ and $b'=(0,0,\ldots 0)$ in $S$. Since $m$ and $n$ are not both equal to zero,
  $b$ and $b'$ are zero-divisors. By the construction of $T'$, we know that $b, b' \notin W$. However, for every $w=(w_1,w_2, \ldots, w_{n+m+1}) \in W$, we have one of the following two possible cases.
  \begin{enumerate}
      \item 
      We have $w_1=0$ and thus (since $w \notin \supp(S)$) there exists $1 \leq j \leq n$ such that $w_{1+m+j} \notin \{0,1\}$. But since all elements of $S_{1+m+j}$ except $1$ are zero-divisors, we have $w_{1+m+j} \in Z(S_{1+m+j})$, so in this case $d(b,w)=d(b',w)=1$.
      \item
      We have $w_1 \neq 0$. Since $w$ is a zero divisor there exists $2 \leq j \leq n+m+1$ such that $w_j$ is a zero-divisor. Again, we have $d(b,w)=d(b',w)=1$.
  \end{enumerate}
  So, we have proved that $b$ and $b'$ have the same representations with respect to $W$, which is a contradiction with the fact that $W$ is a resolving set. Therefore $T' \subsetneq W$, so $|W| \geq |T'|+1=|S|-2^{m+n+1}-(|R|-2)+1=|S|-2^{m+n+1}-|R|+3$. This implies that we have $\dim(\Gamma_1(S)) \geq |S|-2^{n+m+1}-|R|+3$.
   
   Let us now prove that we also have $\dim(\Gamma_1(S)) \leq |S|-2^{n+m+1}-|R|+3$.   
   For every $i=1,2,\ldots,n+m+1$ denote by $\overline{e_i}(x)$ the element of $S$ with $1$ at every component, except the $i$-th component, which is equal to $x$. Let us define $W=T' \cup \{e_1(0)\}$ for the set $T'$ defined as above. 
   Observe that $|W|=|S|-2^{m+n+1}-(|R|-2)+1=|S|-2^{m+n+1}-|R|+3$. Also, 
   Let us prove that $W$ is a resolving set. Choose $b=(r_1,b_1,b_2,\ldots,b_m,z_1,z_2,\ldots,z_n) \neq b'=(r_1',b_1',b_2',\ldots,b_m',z_1',z_2',\ldots,z_n') \in Z(S) \setminus W$ for some $r_1,r_1' \in R$, some $b_i, b_i' \in B_i$ and some $z_j, z_j' \in Z_j$ for every $1 \leq i \leq m, 1 \leq j \leq n$. 
    Suppose firstly that there exists $1 \leq j \leq n$ such that $z_j \neq z_j'$. Since $b, b' \notin W$, we know that $z_j=0$ and $z_j'=1$ or vice versa. Without loss of generality, we can assume the former.
    Now, choose $x \in Z_j \setminus \{0,1\}$ and note that $\overline{e_j}(x) \in T'$ by the construction of the set $T'$. This now implies that $d(b, \overline{e_j}(x))=1$, while $d(b', \overline{e_j}(x)) > 1$.
    Suppose now that $z_j = z_j'$ for all $1 \leq j \leq n$.
    Since $b \neq b'$, we have $r_1 \neq r_1'$ or there exists $1 \leq i \leq m$ such that $b_i \neq b_i'$. 
    Suppose firstly that  $b_i \neq b_i'$ for some $1 \leq i \leq m$. Without any loss of generality, we can assume that $b_i=0$ and $b_i'=1$. Choose $r \in R \setminus \{0\}$ and define $s=e_1(r) e_i(0)$. Again, note that $s \in T'$ and that we now have $d(b,s)=1$ and $d(b',s) > 1$. 
    Finally, suppose $b_i=b_i'$ and $z_j=z_j'$ for all $1 \leq i \leq m$ and $1 \leq j \leq n$. Then we have $r_1 \neq r_1'$ and since $b, b' \notin W$ we either have $r_1=0$ or $r_1'=1$ or vice versa. Again, without loss of generality, assume the former. Observe that $d(b,e_1(0))=1$ and $d(b',e_1(0)) > 1$.
    This implies that $W$ is indeed a resolving set, so $\dim(\Gamma_1(S)) \leq |S|-2^{n+m+1}-|R|+3$. Note also that all elements of $S$ except for elements of the form $(r,1,\ldots,1)$ for some $r \in R \setminus \{0\}$ are zero-divisors, so $|S|=|Z(S)|+|R|-1$ and thus
    $\dim(\Gamma_1(S)) = |S|-2^{n+m+1}-|R|+3=|Z(S)|-2^{m+n+1}+2$, which now proves our assertion.
  \end{enumerate}  
  \end{proof}

Observe that Proposition \ref{semimetricbool}, Theorem \ref{semimetricu} and Theorem \ref{haupt} now cover all possible cases of direct product of antinegative finite commutative semirings with their sets of zero-divisors closed under addition, so we have completely determined the metric dimension in this setting. 

%We can now gather all the above results into one final corollary, which completely covers all possible cases of direct product of antinegative finite commutative semirings with their sets of zero-divisors closed under addidtion. 

%\begin{Corollary}
%\label{final}
%   Let $S=S_1 \times S_2 \times \ldots \times S_n$, where $n \geq 1$ and $S_i$ is an antinegative finite commutative semiring with $Z(S_i)$ closed under addition for every $1 \leq i \leq n$. Then the following holds.
%   \begin{enumerate}
%       \item 
%     If $S$ is isomorphic to $\BB \times \BB$ then $\dim(\Gamma_1(S)) = 1$.
     
%     \item
%     If $S$ is not isomorphic to $\BB \times \BB$ then $\dim(\Gamma_1(S)) = |Z(S)|-2^{n}+m+1$, where $m$ is the number of semirings $S_i$ that are isomorphic to the Boolean semiring $\BB$.
%   \end{enumerate}
%\end{Corollary}

\bigskip

\bigskip

{\bf Statements and Declarations} \\

The author states that there is no conflict of interest.

\bigskip

%-----------------------------------------------------
%-----------------------------------------------------

\bibliographystyle{amsplain}
\bibliography{biblio}

\end{document}